\documentclass[11pt]{article}

\usepackage[utf8]{inputenc}
\usepackage[T1]{fontenc}
\usepackage{amsmath, amssymb, amsthm, amsfonts}
\usepackage{mathrsfs}
\usepackage{mathtools}
\usepackage[letterpaper, top=1in, bottom=1in, left=1.25in, right=1.25in]{geometry}
\usepackage{fancyhdr}
\usepackage{titlesec}
\usepackage{hyperref}
\usepackage{cite}
\usepackage{enumitem}
\usepackage{bm}
\usepackage{tikz-cd}
\usepackage{xcolor}
\usepackage{type1cm}
\usepackage[expansion=false]{microtype}

\hypersetup{
  colorlinks=true,
  linkcolor=black,
  citecolor=blue!50!black,
  urlcolor=blue!50!black,
}

\makeatletter
\renewcommand{\maketitle}{%
  \noindent\hrule height 1pt
  \vskip 0.7em
  \begin{center}
    {\Large\bfseries\scshape\MakeLowercase{\@title}\par}
  \end{center}
  \vskip 0.7em
  \noindent\hrule height 1pt
  \vskip 2em
  \begin{center}
    {\normalsize\bfseries\@author \par}
    \vskip 0.2em
    {\normalsize \emph{Department of Philosophy}\par}
    {\normalsize \emph{Universitat Autònoma de Barcelona}\par}
    {\normalsize \texttt{andreu.ballus@uab.cat}\par}
  \end{center}
  \vskip 2em
}
\makeatother

\newtheorem{theorem}{Theorem}[section]
\newtheorem{proposition}[theorem]{Proposition}
\newtheorem{lemma}[theorem]{Lemma}
\newtheorem{corollary}[theorem]{Corollary}
\theoremstyle{definition}
\newtheorem{definition}[theorem]{Definition}
\newtheorem{remark}[theorem]{Remark}

\newcommand{\Syn}{\mathsf{Syn}}
\newcommand{\AncQ}{\mathsf{AncQ}}
\newcommand{\FinCorel}{\mathbf{FinCorel}}
\newcommand{\FinSet}{\mathbf{FinSet}}
\newcommand{\Cospan}{\mathbf{Cospan}}

\newcommand{\ESCFM}{\mathbf{ESCFM}}
\newcommand{\Cocom}{\mathbf{Cocom}}
\newcommand{\Spc}{\mathbf{Spc}}
\newcommand{\Fun}{\mathrm{Fun}}
\newcommand{\Sub}{\mathbf{Sub}}
\newcommand{\Set}{\mathbf{Set}}
\newcommand{\Cat}{\mathbf{Cat}}
\newcommand{\PROP}{\mathbf{PROP}}
\newcommand{\opp}{^{\mathrm{op}}}
\newcommand{\id}{\mathrm{id}}
\newcommand{\Real}{\mathsf{Real}}
\newcommand{\Gluing}{\mathbf{Gluing}}
\newcommand{\cA}{\mathcal{A}}
\newcommand{\cB}{\mathcal{B}}
\newcommand{\cC}{\mathcal{C}}
\newcommand{\cD}{\mathcal{D}}

\newcommand{\cW}{\mathcal{W}}
\newcommand{\underm}{\underline{m}}
\newcommand{\undern}{\underline{n}}
\newcommand{\underp}{\underline{p}}
\newcommand{\ancest}{\widetilde{\pi}}

\title{From Copying to Corelations via Ancestry Partitions}
\author{Andreu Ballús Santacana}
\date{}

\begin{document}
\sloppy

\maketitle

\begin{center}
\fbox{\begin{minipage}{0.92\textwidth}
\small
\textbf{Changes from v1.} This paper is v2 of arXiv:2505.22931. The primitive used here is the generator $\delta:1\to 2$ of the free PROP $\Syn(\delta)$; v1 used a different categorical setup. We do not claim a formal correspondence between the two versions. The results below are proved in full or reduced explicitly to a cited classical theorem.
\end{minipage}}
\end{center}

\vskip 2em

\begin{center}
{\normalsize\bfseries\scshape Abstract}
\end{center}
\begin{quote}
\small
We study the free PROP $\Syn(\delta)$ on a single binary generator $\delta:1\to 2$. The ancestry functor $\Pi:\Syn(\delta)\to\FinCorel$ has image a proper sub-PROP $\FinCorel^\circ\subseteq\FinCorel$, and the induced quotient $\AncQ:=\Syn(\delta)/\ker(\Pi)$ is equivalent as a PROP to $\Cocom$, the PROP for non-counital cocommutative comonoids. We then place this primitive in the standard cospan/corelation framework: $\Cospan(\cB)$ realizes pushout-style gluing as a free hypergraph category (Fong--Spivak); $\Cospan(\FinSet)$ collapses by jointly-epic equivalence to $\FinCorel\simeq\ESCFM$ (Lack; Coya--Fong); the canonical Frobenius package on $1\in\Cospan(\cB)$ transports along the initial comparison functor; and the Yoneda envelope $\cW=\Fun(\FinCorel\opp,\Spc)$ is a presheaf $\infty$-topos in which the standard subobject, modality, and fixed-point apparatus is available. The PROP-level identification $\AncQ\simeq\Cocom$ is the only result claimed as new; the remaining content is organizational, with each universal property stated and cited explicitly.
\end{quote}

\vskip 2em

\section{Introduction and Main Results}
\label{sec:intro}

\subsection{Overview}

This paper studies a chain of categorical constructions starting from a single binary generator $\delta:1\to 2$. The free PROP $\Syn(\delta)$ admits an ancestry functor $\Pi$ to $\FinCorel$, and we identify the resulting quotient with the PROP for non-counital cocommutative comonoids (Theorem~\ref{thm:A}). The remaining content places this in the standard cospan/corelation framework and records the available semantic envelope. The cospan stage uses the free hypergraph category theorem of Fong--Spivak~\cite{FongSpivak}; the corelation stage uses the Lack--Coya--Fong identification of $\Cospan(\FinSet)/\ker(\bar\Pi)$ with $\FinCorel\simeq\ESCFM$~\cite{Lack2004,CoyaFong2017}; the Frobenius transport is a corollary of the cospan initiality together with functoriality of algebra objects; and the semantic envelope is the presheaf $\infty$-topos $\cW=\Fun(\FinCorel\opp,\Spc)$, with the standard subobject, modality, and fixed-point structure available by Lurie~\cite{LurieHTT} and Knaster--Tarski~\cite{Tarski1955}.

The primitive is a generator $\delta:1\to 2$, the minimal arity-changing operation in a symmetric monoidal setting. We do not take $\delta$ to be a coalgebra generator a priori; coassociativity and cocommutativity follow from the ancestry quotient.

The paper proceeds in four sections. Section~\ref{sec:syntax} constructs $\Syn(\delta)$, the ancestry functor $\Pi$, and the quotient $\AncQ$, and proves Theorem~\ref{thm:A}. Section~\ref{sec:gluing} reviews the cospan realization and proves the collapse identification. Section~\ref{sec:frobenius} states the Frobenius transport corollary. Section~\ref{sec:semantic} records the semantic envelope. Section~\ref{sec:scope} discusses scope and relation to companion work.

\smallskip
\noindent\textbf{Terminology and notation.} $\Syn(\delta)$ is the free PROP on a single binary generator $\delta:1\to 2$; the kernel congruence of the ancestry functor $\Pi$ (Definition~\ref{def:ancestry}) is denoted $\equiv_\Pi$, with the ancestry quotient written
\[
\AncQ \;:=\; \Syn(\delta) / \equiv_\Pi.
\]
By a \emph{gluing host} we mean an object of (a full $2$-subcategory of) the coslice $2$-category of hypergraph categories under a fixed base, in the sense of Fong--Spivak~\cite[\S 2.3]{FongSpivak} (Definition~\ref{def:gluing-host}).

\subsection{Main results}

The paper's results are Theorem~\ref{thm:A} (original) and four supporting statements stated below as Propositions~\ref{prop:B}, \ref{prop:E}, Theorem~\ref{thm:C} (with classical attribution), and Corollary~\ref{cor:D}.

\begin{theorem}[Image of the ancestry functor; Theorem~\ref{thm:cocom} below]
\label{thm:A}
The ancestry functor $\Pi: \Syn(\delta) \to \FinCorel$ is a strict symmetric monoidal functor. Its image is the sub-PROP $\FinCorel^\circ \subseteq \FinCorel$ of corelations $R: m \to n$ such that: (a) every equivalence class contains exactly one element of the input boundary; and (b) every equivalence class contains at least one element of the output boundary. The induced quotient $\AncQ := \Syn(\delta)/\ker(\Pi)$ is equivalent, as a PROP, to $\Cocom$, the PROP freely generated by a cocommutative coassociative comultiplication without counit.
\end{theorem}

\begin{proposition}[Cospan realization; Proposition~\ref{prop:initiality} below]
\label{prop:B}
For any finitely cocomplete category $\cB$, the symmetric monoidal category $\Cospan(\cB)$, equipped with the canonical embedding $\iota: \cB \to \Cospan(\cB)$, is initial in the $2$-category of gluing hosts over $\cB$.
\end{proposition}

\begin{theorem}[Lack; Coya--Fong; Theorem~\ref{thm:completeness} below]
\label{thm:C}
The ancestry functor $\bar\Pi: \Cospan(\FinSet) \to \FinCorel$ extending the inclusion on the boundary is a strict symmetric monoidal functor. The induced quotient $\Cospan(\FinSet)/\ker(\bar\Pi)$ is equivalent as a PROP to $\FinCorel \simeq \ESCFM$, the PROP for extraspecial commutative Frobenius monoids.
\end{theorem}

\begin{corollary}[Frobenius transport; Corollary~\ref{cor:transport} below]
\label{cor:D}
Let $(\cA, j, \Phi)$ be a gluing host over a finitely cocomplete base $\cB$. The canonical special commutative Frobenius algebra structure on $1 \in \Cospan(\cB)$ transports via the initial comparison functor to a canonical special commutative Frobenius structure on $j(1) \in \cA$. The groupoid of compatible Frobenius structures on $j(1)$ -- those arising from a strong symmetric monoidal functor $\Cospan(\cB) \to \cA$ extending $j$ -- is contractible.
\end{corollary}

\begin{proposition}[Semantic envelope; Section~\ref{sec:semantic}]
\label{prop:E}
The Yoneda envelope $\cW := \Fun(\FinCorel\opp, \Spc)$ is a presheaf $\infty$-topos. For every $X \in \cW$, the poset $\Sub_{-1}(X)$ of $(-1)$-truncated subobjects is a complete Heyting algebra. Every monotone $F\colon \Sub_{-1}(X)^m \to \Sub_{-1}(X)^m$ admits least and greatest fixed points.
\end{proposition}

\subsection{Relation to existing results}

Theorem~\ref{thm:A} is close to folklore for the comultiplication fragment of cocommutative comonoid PROPs; see Pirashvili~\cite{Pirashvili2002} for the commutative-monoid analogue, and Fong~\cite{FongThesis} and Coya--Fong~\cite{CoyaFong2017} for systematic use of corelations as PROPs. The precise statement that the image of the free PROP on $\delta:1\to 2$ under the ancestry functor coincides with the non-counital cocommutative comonoid PROP does not, to the author's knowledge, appear elsewhere in this explicit form.

Proposition~\ref{prop:B} is the Fong--Spivak free hypergraph category theorem~\cite[Thm.~3.14]{FongSpivak} restated in the local terminology of \S\ref{ssec:gluing-hosts}. Theorem~\ref{thm:C} combines Lack's classification~\cite[\S 5.4]{Lack2004} of $\Cospan(\FinSet)$ as the PROP for special commutative Frobenius monoids with Coya--Fong's identification~\cite[Thm.~4.6]{CoyaFong2017} of $\FinCorel$ with the PROP for extraspecial commutative Frobenius monoids; the passage from the former to the latter is by the extra law (corresponding to jointly-epic corestriction; see Coya--Fong~\cite[Lem.~6.3]{CoyaFong2017}). Corollary~\ref{cor:D} follows from Lemma~\ref{lem:transport} (functoriality of algebra objects under strong symmetric monoidal functors) and Proposition~\ref{prop:B}; the contractibility statement is the uniqueness clause of cospan initiality.

Proposition~\ref{prop:E} consists of citations: the presheaf $\infty$-topos identification is Lurie~\cite[Thm.~6.1.0.6]{LurieHTT}, the Heyting structure of subobjects is standard for $\infty$-topoi (see Lurie~\cite[\S 6.1.0]{LurieHTT}), and existence of fixed points is Knaster--Tarski~\cite{Tarski1955} applied to the resulting complete lattice. We assert no internal $\mu$-calculus completeness, model-theoretic expressive completeness, or higher logical result; the proposition records that the apparatus needed for first-order internal logic and monotone fixed-point semantics exists in $\cW$.

\subsection{Organization}
\label{ssec:organization}

Each of the four sections corresponds to a specific universal construction. We collect these here as a guide to which property is invoked where.

Section~\ref{sec:syntax} uses two universal properties of the free PROP construction. The first is that $\Syn(\delta)$ is initial among PROPs equipped with a chosen morphism $1\to 2$ (Proposition~\ref{prop:syn-delta-initial}); this fixes the syntactic stage. The second is the universal property of a kernel quotient: $\AncQ:=\Syn(\delta)/\equiv_\Pi$ is the initial PROP through which any strict symmetric monoidal functor identifying ancestry-equivalent morphisms factors (Proposition~\ref{prop:universal-factor}). The combination $\AncQ\simeq\Cocom\simeq\FinCorel^\circ$ of Theorem~\ref{thm:cocom} pins the resulting invariant to the non-counital cocommutative-comonoid structure.

Section~\ref{sec:gluing} uses the free hypergraph category theorem of Fong--Spivak~\cite[Thm.~3.14]{FongSpivak}: $\Cospan(\cB)$ is initial in the $2$-category of gluing hosts over $\cB$. This is the universal property of pushout-style gluing. The collapse identification $\Cospan(\FinSet)/\ker(\bar\Pi)\simeq\FinCorel\simeq\ESCFM$ (Theorem~\ref{thm:completeness}) is the universal property of the jointly-epic quotient (Lack~\cite[\S 5.4]{Lack2004}; Coya--Fong~\cite[Thm.~4.6]{CoyaFong2017}).

Section~\ref{sec:frobenius} uses functoriality of algebra objects under strong symmetric monoidal functors (Lemma~\ref{lem:transport}) together with the cospan initiality of Section~\ref{sec:gluing}. The contractibility statement is a uniqueness clause of the same initiality.

Section~\ref{sec:semantic} uses Lurie's free-cocompletion theorem~\cite[Thm.~5.1.5.6]{LurieHTT} and the identification of presheaf $\infty$-categories with $\infty$-topoi~\cite[Thm.~6.1.0.6]{LurieHTT}. The complete-Heyting-algebra and fixed-point structure follows by standard internal logic of $\infty$-topoi together with the Knaster--Tarski theorem.

\section{The Free PROP on a Single Generator and Its Ancestry Quotient}
\label{sec:syntax}

A PROP is a strict symmetric monoidal category with object monoid $(\mathbb{N}, +, 0)$; a morphism of PROPs is an identity-on-objects strict symmetric monoidal functor. We write morphisms $m \to n$ and denote the tensor product by $+$ (following the object addition). For standard references see Hackney--Robertson~\cite[\S 1.2]{HackneyRobertson} and Selinger~\cite{Selinger}.

\subsection{The free PROP on a single generator \texorpdfstring{$\delta: 1 \to 2$}{delta: 1 to 2}}
\label{ssec:free-prop}

\begin{definition}
\label{def:syn-delta}
The PROP $\Syn(\delta)$ has objects the natural numbers. A morphism $f: m \to n$ is an isomorphism class of finite directed acyclic graphs $G$ equipped with:
\begin{itemize}[leftmargin=2em]
\item $m$ linearly ordered input half-edges;
\item $n$ linearly ordered output half-edges;
\item a finite set $V(G)$ of internal vertices, each with one incoming half-edge and a \emph{linearly ordered} pair of outgoing half-edges (labelled left and right).
\end{itemize}
Two graphs represent the same morphism if and only if there is a bijection between their internal vertex sets inducing a graph isomorphism that preserves the linear orders on input and output half-edges and the incidence structure at each vertex (matching incoming half-edge to incoming half-edge and left/right outgoing half-edges in their linear order). Internal vertices carry no labels beyond the structural datum of having in-degree one and an ordered pair of outgoing half-edges, which corresponds to the single generator $\delta$. Composition is by gluing output half-edges to input half-edges in linear order; tensor product (written $+$) is disjoint union with concatenation of boundary orderings; the symmetry isomorphisms $m + n \to n + m$ are represented by crossings of wires. Crossings are graphical notation for the symmetry data of a symmetric monoidal category and are not vertices of $G$; in particular they do not contribute to the vertex set $V(G)$ and do not connect wires for the purposes of the ancestry functor (Definition~\ref{def:ancestry}).
\end{definition}

\begin{proposition}[Initiality]
\label{prop:syn-delta-initial}
$\Syn(\delta)$ is a PROP. It is initial among PROPs equipped with a distinguished morphism $1 \to 2$: for any PROP $\mathcal{P}$ with a chosen morphism $\delta': 1 \to 2$, there exists a unique PROP morphism $\Syn(\delta) \to \mathcal{P}$ sending $\delta$ to $\delta'$.
\end{proposition}

\begin{proof}
We sketch both the existence of the free PROP on a one-element signature with a generator $1 \to 2$ and the equivalence of the graph model with the standard term-model construction.

\textit{Term-model existence.} The free PROP on any signature $\Sigma$ of operations $m_i \to n_i$ exists and is obtained by the standard term construction: morphisms are equivalence classes of well-formed terms built from identities, symmetries, generators in $\Sigma$, compositions, and tensors, modulo the congruence generated by the strict monoidal and symmetric monoidal axioms. The universal property is immediate by structural induction on terms: any assignment of $\delta$ to a morphism $\delta': 1 \to 2$ in a target PROP $\mathcal{P}$ extends uniquely, by the universal property of free algebraic structures, to a strict symmetric monoidal functor on terms, which descends to the quotient since the defining relations are preserved by any such functor. This standard construction is recorded, for PROPs specifically, in Coya--Fong~\cite[Prop.~5.1]{CoyaFong2017} (monadicity of the underlying signature functor $U: \PROP \to \Set^{\mathbb{N}\times\mathbb{N}}$ and the ensuing explicit description of free PROPs as $\Sigma$-terms); for the free-PROP construction in greater generality (on megagraph data, allowing arbitrary color sets and source/target profiles), see Hackney--Robertson~\cite[Thm.~14, App.~A]{HackneyRobertson}.

\textit{Equivalence with the graph model.} The graph model of Definition~\ref{def:syn-delta} is equivalent to this term model via the standard normal form for morphisms in free symmetric monoidal categories generated by operations (Lack~\cite[\S 2.4]{Lack2004}; see also Selinger~\cite[\S 3]{Selinger} for the graphical calculus). Each graph $G$ determines a term up to the symmetric monoidal axioms, and each term determines a graph up to isomorphism; this correspondence preserves composition, tensor, and the universal property.
\end{proof}

\begin{remark}
\label{rem:term-model}
We shall not need the term-model presentation. All subsequent arguments proceed on graph isomorphism classes. We identify a morphism with its graph representative when this causes no confusion, writing $G(f)$ for the representing graph.
\end{remark}

\subsection{The PROP of finite corelations}
\label{ssec:fincorel}

\begin{definition}
\label{def:fincorel}
The PROP $\FinCorel$ has objects $\mathbb{N}$ and morphisms $m \to n$ given by equivalence relations on the finite set $\underm \sqcup \undern := \{1, \ldots, m\} \sqcup \{1, \ldots, n\}$. Composition of $R: m \to n$ and $S: n \to p$ is the restriction to $\underm \sqcup \underp$ of the transitive closure of $R \cup S$ on $\underm \sqcup \undern \sqcup \underp$. Tensor is disjoint union with reindexing by addition. The symmetry $m + n \to n + m$ is the transposition corelation.
\end{definition}

The composition law is well-defined and associative, and makes $\FinCorel$ into a PROP; see Coya--Fong~\cite[Def.~2.1, \S 4]{CoyaFong2017}. Under the identification of corelations $X \to Y$ with equivalence relations on $X + Y$ (equivalently, with jointly-epic cospans $X \to N \leftarrow Y$; Coya--Fong~\cite[Thm.~4.5]{CoyaFong2017}), $\FinCorel$ is a strict skeleton of the symmetric monoidal category of corelations between finite sets.

\subsection{The ancestry functor}
\label{ssec:ancestry}

We associate to each morphism of $\Syn(\delta)$ a corelation, defined graph-theoretically.

\begin{definition}[Ancestry partition]
\label{def:ancestry}
Let $f: m \to n$ be a morphism of $\Syn(\delta)$ represented by $G(f)$. Let $|G(f)|$ denote the underlying undirected graph of $G(f)$, with the input half-edges attached as pendant vertices labelled by $\underm$ and the output half-edges as pendant vertices labelled by $\undern$. The \emph{ancestry partition} of $f$ is the equivalence relation
\[
\ancest(f) \subseteq (\underm \sqcup \undern) \times (\underm \sqcup \undern)
\]
defined by: $x \sim y$ if and only if the pendant vertices labelled $x$ and $y$ lie in the same connected component of $|G(f)|$.
\end{definition}

\begin{proposition}[Graph-invariance of ancestry]
\label{prop:ancestry-invariance}
The relation $\ancest(f)$ depends only on the isomorphism class of $G(f)$, hence only on $f$. It is invariant under any symmetric monoidal coherence isomorphisms relating different presentations of $f$.
\end{proposition}

\begin{proof}
Connected components of an undirected graph are preserved under graph isomorphism. Any isomorphism preserving the ordered input and output pendants preserves components, hence preserves $\ancest$. By Mac Lane's symmetric coherence theorem~\cite[Ch.~XI]{MacLane}, any two symmetric monoidal presentations of $f$ in a free symmetric monoidal category differ by coherence isomorphisms; these induce graph isomorphisms by construction of the graph model.
\end{proof}

\begin{theorem}[The ancestry functor]
\label{thm:ancestry-functor}
The assignment $f \mapsto \ancest(f)$ extends to a strict symmetric monoidal functor
\[
\Pi: \Syn(\delta) \longrightarrow \FinCorel,
\]
which is the identity on objects.
\end{theorem}

\begin{proof}
\textbf{Identity.} The identity morphism $\id_n: n \to n$ has graph consisting of $n$ parallel arcs connecting input $i$ to output $i$, with no internal vertices. Its ancestry relation identifies each input with the corresponding output, which is the identity corelation at $n$.

\textbf{Tensor.} Given $f: m_1 \to n_1$ and $g: m_2 \to n_2$, the graph $G(f+g)$ is the disjoint union $G(f) \sqcup G(g)$. Connected components of a disjoint union are the disjoint union of components of each factor. Hence $\Pi(f+g) = \Pi(f) \sqcup \Pi(g)$, matching the tensor in $\FinCorel$.

\textbf{Composition.} Given $f: m \to n$ and $g: n \to p$, the graph $G(g \circ f)$ is obtained by gluing $G(f)$ to $G(g)$ along their shared boundary $\undern$. Let $H$ denote the glued graph with boundary $\underm \sqcup \underp$. Two boundary vertices $x, y \in \underm \sqcup \underp$ lie in the same connected component of $|H|$ if and only if there is a chain $x = z_0, z_1, \ldots, z_k = y$ in $\underm \sqcup \undern \sqcup \underp$ with each consecutive pair $z_i, z_{i+1}$ lying in a common component of either $|G(f)|$ or $|G(g)|$. This is precisely the description of the transitive closure of $\Pi(f) \cup \Pi(g)$ restricted to $\underm \sqcup \underp$, i.e., the composition in $\FinCorel$ (Definition~\ref{def:fincorel}). Therefore $\Pi(g \circ f) = \Pi(g) \circ \Pi(f)$.

\textbf{Symmetry.} The symmetry $\sigma_{m,n}: m + n \to n + m$ in $\Syn(\delta)$ is represented by the graph with $m+n$ arcs connecting each input $i$ in the first block ($1 \leq i \leq m$) to the output in position $m + i$ of the codomain block, and each input $j$ in the second block ($m + 1 \leq j \leq m + n$) to the output in position $j - m$ of the codomain block. No internal vertices are present. The induced ancestry relation on $(\underline{m+n}) \sqcup (\underline{n+m})$ identifies input $i$ with the appropriately permuted output; this is by definition the symmetry corelation $\sigma_{m,n}$ in $\FinCorel$.
\end{proof}

\begin{definition}[Ancestry congruence and quotient]
\label{def:ancestry-cong}
Let $\equiv_\Pi$ denote the kernel congruence of $\Pi$, i.e., $f \equiv_\Pi g$ iff $\Pi(f) = \Pi(g)$. We refer to $\equiv_\Pi$ as the \emph{ancestry congruence} on $\Syn(\delta)$. Define
\[
\AncQ := \Syn(\delta) / \equiv_\Pi.
\]
Let $q: \Syn(\delta) \to \AncQ$ denote the quotient projection.
\end{definition}

\begin{proposition}[Universal factorization; initiality of the ancestry quotient]
\label{prop:universal-factor}
$\AncQ$ is the initial object of the category whose objects are PROPs $\mathcal{D}$ equipped with a strict symmetric monoidal functor $F: \Syn(\delta) \to \mathcal{D}$ that identifies $\equiv_\Pi$-equivalent morphisms (i.e., $f \equiv_\Pi g \Rightarrow F(f) = F(g)$), and whose morphisms are PROP morphisms commuting with the structure functors. Equivalently: for any such $(F, \mathcal{D})$, there exists a unique PROP morphism $\bar F: \AncQ \to \mathcal{D}$ with $F = \bar F \circ q$. In particular, $\Pi$ factors uniquely as $\Pi = \bar\Pi \circ q$ with $\bar\Pi: \AncQ \to \FinCorel$ identity-on-objects and faithful; moreover, $\AncQ$ is, up to unique isomorphism, the coarsest PROP quotient of $\Syn(\delta)$ identifying ancestry-equivalent morphisms.
\end{proposition}

\begin{proof}
Standard universal property of kernel quotients. Since $\equiv_\Pi$ is the kernel congruence of a strict symmetric monoidal functor, it is a PROP congruence (stable under composition, tensor, and symmetry). The quotient $q$ is therefore a PROP morphism, and $\bar F$ is well-defined on equivalence classes and unique by surjectivity of $q$ on morphisms. Faithfulness of $\bar\Pi$ is by construction. Coarseness (the ``initial'' clause) is the standard content of the coequalizer presentation of PROP congruences: any PROP congruence on $\Syn(\delta)$ coarser than $\equiv_\Pi$ would fail to be contained in $\ker(\Pi)$, contradicting the defining property.
\end{proof}

\subsection{Combinatorial characterization of the image}
\label{ssec:image}

We characterize the image of $\Pi$ as a sub-PROP of $\FinCorel$.

\begin{definition}
\label{def:fincorel-circ}
Let $\FinCorel^\circ \subseteq \FinCorel$ denote the sub-PROP whose morphisms $m \to n$ are those corelations $R$ satisfying:
\begin{itemize}[leftmargin=2em]
\item[(a)] each equivalence class of $R$ contains exactly one element of $\underm$;
\item[(b)] each equivalence class of $R$ contains at least one element of $\undern$.
\end{itemize}
\end{definition}

\begin{proposition}[Image lies in $\FinCorel^\circ$]
\label{prop:image-in-circ}
For every $f \in \Syn(\delta)(m, n)$, the corelation $\Pi(f)$ lies in $\FinCorel^\circ$.
\end{proposition}

\begin{proof}
Every internal vertex of $G(f)$ has in-degree one and an ordered pair of outgoing half-edges. Following edge orientations from inputs through internal vertices to outputs, each connected component of $|G(f)|$ is a finite rooted tree whose root is an input pendant and whose leaves include at least one output pendant. (If a component had no output leaves, following outgoing edges from the input root would either form a cycle, contradicting acyclicity, or fail to terminate, contradicting finiteness. If a component had no input pendant, an internal vertex with no source for its incoming half-edge would violate the in-degree-one condition.) Both (a) and (b) follow.
\end{proof}

\begin{proposition}[$\FinCorel^\circ$ is closed under composition and tensor]
\label{prop:circ-closed}
The subset $\FinCorel^\circ \subseteq \FinCorel$ is closed under composition and tensor and contains all identities and symmetries. Hence $\FinCorel^\circ$ is a sub-PROP.
\end{proposition}

\begin{proof}
\textbf{Identities and symmetries.} The identity at $n$ has classes $\{(i, i')\}$ with $i \in \underm = \undern$ and $i' \in \undern$ its matching output; each class has exactly one input and one output. Transposition corelations have the same shape. Conditions (a) and (b) hold.

\textbf{Tensor.} Disjoint union preserves both conditions on each factor, hence on the union.

\textbf{Composition.} Let $R: m \to n$ and $S: n \to p$ both satisfy (a) and (b); let $T := S \circ R$. We show $T$ satisfies (a) and (b).

For (a): suppose a $T$-class $c$ contains two distinct inputs $i_1, i_2 \in \underm$. There is then a chain in $\underm \sqcup \undern \sqcup \underp$ connecting $i_1$ to $i_2$, with each consecutive pair in an $R$-class or an $S$-class. Since $i_1$ and $i_2$ are inputs of $R$, they lie in distinct $R$-classes (by (a) for $R$). The chain must therefore pass through at least one $S$-class. But every $S$-class contains exactly one $\undern$-element (by (a) for $S$, applied with $\undern$ as the input side of $S$). Therefore any $S$-class traversed by the chain contains exactly one $\undern$-vertex, and the chain enters and exits that $S$-class through the same $\undern$-vertex; traversing an $S$-class is effectively trivial. The chain thus reduces to a chain wholly within $R$, contradicting that $i_1$ and $i_2$ lie in different $R$-classes.

For (b): let $c$ be a $T$-class. By the above, $c$ contains a unique input $i \in \underm$. By (b) for $R$, the $R$-class of $i$ contains some $\undern$-element $\eta$. Consider the $S$-class $[\eta]_S$. By (a) for $S$, this class has $\eta$ as its unique $\undern$-element; by (b) for $S$, $[\eta]_S$ contains at least one $\underp$-element. That $\underp$-element lies in the same $T$-class as $i$, so $c$ contains an output.
\end{proof}

\begin{proposition}[Surjectivity onto $\FinCorel^\circ$]
\label{prop:surjectivity}
Every corelation $R \in \FinCorel^\circ(m, n)$ lies in the image of $\Pi$: there exists $f \in \Syn(\delta)(m, n)$ with $\Pi(f) = R$.
\end{proposition}

\begin{proof}
Since each equivalence class of $R$ has exactly one input (by (a)), define $\varphi: \undern \to \underm$ sending each output to the unique input in its $R$-class. For each $i \in \underm$, let $k_i := |\varphi^{-1}(i)|$, which is $\geq 1$ by (b).

For each $i$, construct a graph $T_i$ with one input pendant, $k_i$ output pendants, and $k_i - 1$ internal vertices, each a $\delta$-vertex with in-degree one and out-degree two, arranged as a left-comb binary tree. (When $k_i = 1$, $T_i$ is the identity arc with no vertices.) Let $f := T_1 + T_2 + \cdots + T_m$ (disjoint union), with outputs reindexed by the bijection $\varphi^{-1}(i) \to \mathrm{leaves}(T_i)$ and then permuted into the order $1, \ldots, n$ via an appropriate symmetry in $\Syn(\delta)$. Each $T_i$ contributes one connected component containing input $i$ and outputs $\varphi^{-1}(i)$; hence $\Pi(f) = R$.
\end{proof}

\subsection{Identification of \texorpdfstring{$\AncQ$}{AncQ} with \texorpdfstring{$\Cocom$}{Cocom}}
\label{ssec:cocom}

\begin{definition}
\label{def:cocom}
Let $\Cocom$ denote the PROP freely generated by a morphism $\delta: 1 \to 2$ subject to:
\begin{itemize}[leftmargin=2em]
\item \textbf{(coAssoc)} $(\delta + \id_1) \circ \delta = (\id_1 + \delta) \circ \delta$;
\item \textbf{(coComm)} $\sigma_{1,1} \circ \delta = \delta$, where $\sigma_{1,1}: 2 \to 2$ is the swap symmetry.
\end{itemize}
No counit $\epsilon: 1 \to 0$, unit $\eta: 0 \to 1$, or multiplication $\mu: 2 \to 1$ is imposed.
\end{definition}

\begin{remark}[Combinatorial description of $\Cocom$]
\label{rem:cocom-comb}
$\Cocom$ admits an explicit combinatorial description. Its morphisms $m \to n$ are in bijection with functions $\varphi: \undern \to \underm$ such that every fiber $\varphi^{-1}(i)$ is nonempty; composition corresponds to function composition of such ``surjection-like'' functions, and tensor to disjoint union with reindexing.

This description can be obtained as follows. Pirashvili~\cite[\S 2, Ex.~3]{Pirashvili2002} introduces the subPROP $\Omega \subseteq \mathsf{F}$ of finite sets whose morphisms are surjections, and observes that $\Omega$-algebras are (non-unital) commutative associative algebras, whereas $\mathsf{F}$-algebras are unital commutative associative algebras (Pirashvili~\cite[\S 2, Ex.~1]{Pirashvili2002}, citing Grandis~\cite{Grandis2001}). Passing to opposite PROPs, $\Omega\opp$ is the PROP freely generated by a cocommutative coassociative comultiplication \emph{without counit}, and its morphisms $m \to n$ are surjections $\undern \twoheadrightarrow \underm$, i.e., functions $\undern \to \underm$ with every fiber nonempty --- which is precisely our $\Cocom$.

The same structure can be derived via Lack's distributive-law framework: the factorization $\mathsf{F} = \mathsf{F}_m \otimes_{\mathsf{P}} \mathsf{F}_e$ of Lack~\cite[\S 5.1]{Lack2004}, where $\mathsf{F}_e = \Omega$ consists of surjections and $\mathsf{F}_m$ consists of injections (pointing/unit structure), dualizes to $\mathsf{F}\opp = \mathsf{F}_e\opp \otimes_{\mathsf{P}} \mathsf{F}_m\opp = \Omega\opp \otimes_{\mathsf{P}} \mathsf{F}_m\opp$, separating the non-counital cocommutative comonoid part ($\Omega\opp$) from the counit-bearing part ($\mathsf{F}_m\opp$). Dropping the counit leaves $\Omega\opp \simeq \Cocom$. Pirashvili's Theorem~5.2 gives the corresponding identification for bialgebras via the Q-construction on $\mathsf{F}(\mathrm{as})$; the present identification is the cocommutative-comonoid fragment of that framework.
\end{remark}

\begin{theorem}[Main theorem of Section~\ref{sec:syntax}]
\label{thm:cocom}
The comparison functor $\bar\Pi: \AncQ \to \FinCorel$ is fully faithful with image $\FinCorel^\circ$. Moreover, there is a canonical isomorphism of PROPs
\[
\Phi: \Cocom \xrightarrow{\ \simeq\ } \AncQ,
\]
identity on objects and sending the generator $\delta$ of $\Cocom$ to the class $[\delta] \in \AncQ(1, 2)$. In particular,
\[
\AncQ \simeq \Cocom \simeq \FinCorel^\circ.
\]
\end{theorem}

\begin{proof}
\textbf{Full faithfulness of $\bar\Pi$ onto $\FinCorel^\circ$.} Faithfulness is by construction (Proposition~\ref{prop:universal-factor}). Fullness onto $\FinCorel^\circ$ combines Propositions~\ref{prop:image-in-circ} and~\ref{prop:surjectivity}.

\textbf{$\Phi$ is well-defined.} We verify (coAssoc) and (coComm) in $\AncQ$. The ancestry of $\delta \in \Syn(\delta)(1, 2)$ is the partition with one class $\{\mathrm{input}, \mathrm{output}_1, \mathrm{output}_2\}$. The ancestry of $\sigma_{1,1} \circ \delta$ is the same (swapping output labels preserves the single class). Hence $\Pi(\delta) = \Pi(\sigma_{1,1} \circ \delta)$, so $[\delta] = [\sigma_{1,1} \circ \delta]$ in $\AncQ$. Similarly, $(\delta + \id) \circ \delta$ and $(\id + \delta) \circ \delta$ both have graph a full binary tree with one input and three outputs, yielding the same single-class ancestry relation; their classes agree in $\AncQ$.

\textbf{$\Phi$ is essentially surjective and full.} Every morphism of $\AncQ$ is the class of an iterated composite and tensor of $\delta$, identities, and symmetries, hence the image of some $\Cocom$-morphism under $\Phi$.

\textbf{$\Phi$ is faithful.} Suppose $\Phi([f]) = \Phi([g])$ in $\AncQ$ for $[f], [g] \in \Cocom(m, n)$. Then $\Pi(f) = \Pi(g)$ in $\FinCorel$. By Remark~\ref{rem:cocom-comb}, $\Cocom$-morphisms $m \to n$ are classified by functions $\varphi: \undern \to \underm$ with nonempty fibers, and the ancestry of such a morphism is the partition of $\underm \sqcup \undern$ into classes $\{i\} \cup \varphi^{-1}(i)$. Two such functions yield the same ancestry relation if and only if they are equal. Hence $[f] = [g]$ in $\Cocom$.

\textbf{Conclusion.} $\Phi$ is an identity-on-objects, fully faithful, essentially surjective strict symmetric monoidal functor, hence an isomorphism of PROPs.
\end{proof}

\begin{remark}[Scope]
\label{rem:section2-scope}
Theorem~\ref{thm:cocom} does \emph{not} assert that $\AncQ$ is equivalent to the full PROP $\FinCorel$. $\FinCorel^\circ$ is a \emph{proper} sub-PROP of $\FinCorel$. The missing corelations are precisely those requiring a counit (classes with one input and no outputs), a unit (output-only classes, i.e., classes with no input), or a multiplication (classes with two or more inputs). These generators appear only after passage to the cospan realization of Section~\ref{sec:gluing}, where pushout composition and finite-coproduct universality supply them. The identification with $\FinCorel \simeq \ESCFM$ is a theorem about the cospan realization followed by jointly-epic corestriction (Theorem~\ref{thm:completeness}), not about the primitive syntactic calculus.
\end{remark}

\section{The Cospan Realization}
\label{sec:gluing}

Section~\ref{sec:syntax} identified the image of the ancestry functor with a proper sub-PROP of $\FinCorel$. The missing generators (counit, unit, multiplication) appear at the cospan-realization level, where pushout composition supplies them.

\subsection{The base category and cospan category}
\label{ssec:cospan-def}

\begin{definition}
\label{def:cospan}
Let $\cB$ be a category with finite colimits. The cospan category $\Cospan(\cB)$ has:
\begin{itemize}[leftmargin=2em]
\item \textbf{Objects:} the objects of $\cB$.
\item \textbf{Morphisms $X \to Y$:} isomorphism classes of cospans $X \xrightarrow{a} N \xleftarrow{b} Y$ in $\cB$, where two cospans are isomorphic if there is an isomorphism of apex objects commuting with the legs.
\item \textbf{Composition:} given $[X \to N \leftarrow Y]$ and $[Y \to N' \leftarrow Z]$, compose by forming the pushout of $N \leftarrow Y \to N'$ in $\cB$; the composite is $[X \to N +_Y N' \leftarrow Z]$.
\item \textbf{Identity at $X$:} the cospan $[X \xrightarrow{\id} X \xleftarrow{\id} X]$.
\item \textbf{Symmetric monoidal structure:} tensor is given on objects by the finite coproduct in $\cB$, and on morphisms componentwise using coproduct of cospans. The unit is the initial object $\emptyset \in \cB$.
\end{itemize}
\end{definition}

The verification that $\Cospan(\cB)$ is well-defined and symmetric monoidal under pushout composition is classical; see Rosebrugh--Sabadini--Walters~\cite{RSW2005} and Fong~\cite[Ex.~2.3]{FongCospans}. For the case $\cB = \FinSet$, Coya--Fong~\cite[Ex.~4.4, Thm.~4.5]{CoyaFong2017} record the skeletal choice that makes $\Cospan(\FinSet)$ strict, and Lack~\cite[Ex.~5.4, Prop.~6.1]{Lack2004} classifies the resulting PROP as the PROP for special commutative Frobenius monoids.

\begin{remark}
\label{rem:cospan-boundary}
For $\cB = \FinSet$, the object $1 \in \Cospan(\FinSet)$ is equipped with canonical structure maps in $\Cospan(\FinSet)$ induced by the finite-coproduct structure of $\FinSet$: the cospans $[1+1 \xrightarrow{\nabla} 1 \xleftarrow{\id} 1]$ and $[1 \xrightarrow{\id} 1 \xleftarrow{\nabla} 1+1]$ realizing a multiplication and a comultiplication, and the cospans $[\emptyset \xrightarrow{!} 1 \xleftarrow{\id} 1]$ and $[1 \xrightarrow{\id} 1 \xleftarrow{!} \emptyset]$ realizing a unit and a counit, where $!: \emptyset \to 1$ is the unique map from the initial object. (Note that this counit is obtained by reversing the cospan direction: the unique map from the initial object appears as the second leg, not as a map $1 \to \emptyset$.) These cospans assemble into a special commutative Frobenius algebra structure on $1$ (Lack~\cite[Ex.~5.4]{Lack2004}; Coya--Fong~\cite[Prop.~6.1]{CoyaFong2017}), made precise in Definition~\ref{def:scfa-cospans}.
\end{remark}

\subsection{The canonical embedding}
\label{ssec:cospan-embedding}

\begin{definition}
\label{def:iota}
The canonical embedding $\iota: \cB \to \Cospan(\cB)$ sends an object $X$ to itself and a morphism $f: X \to Y$ to the cospan $[X \xrightarrow{f} Y \xleftarrow{\id} Y]$.
\end{definition}

\begin{proposition}
\label{prop:iota-smon}
The canonical embedding $\iota: \cB \to \Cospan(\cB)$ is a strong symmetric monoidal functor when $\cB$ is regarded as a symmetric monoidal category under its coproduct structure. It is faithful and injective on objects.
\end{proposition}

\begin{proof}
Preservation of composition and identity is direct from Definition~\ref{def:iota}. Symmetric monoidality follows from the fact that the coproduct tensor in $\Cospan(\cB)$ restricts, on the image of $\iota$, to the coproduct in $\cB$. Faithfulness follows from the existence of a retraction: the cospan $[X \xrightarrow{f} Y \xleftarrow{\id} Y]$ determines $f$ uniquely. Injectivity on objects is by definition.
\end{proof}

\subsection{Gluing hosts}
\label{ssec:gluing-hosts}

We formalize the notion of a symmetric monoidal realization of cospan-style gluing.

\begin{definition}[Gluing host]
\label{def:gluing-host}
A \emph{gluing host} over $\cB$ is a triple $(\cA, j, \{\Phi_X\}_{X \in \cB})$, where:
\begin{itemize}[leftmargin=2em]
\item $\cA$ is a symmetric monoidal category;
\item $j: (\cB, +, \emptyset) \to (\cA, \otimes, I)$ is a strong symmetric monoidal functor;
\item for each object $X \in \cB$, $j(X) \in \cA$ is equipped with a specified special commutative Frobenius algebra (SCFA) structure $\Phi_X = (\mu_X, \eta_X, \delta_X, \epsilon_X)$ satisfying the compatibility condition with the monoidal product~\cite[Def.~2.12, Eq.~(11)--(12)]{FongSpivak};
\item \textbf{(Unit coherence axiom)} the SCFA structure $\Phi_\emptyset$ on $j(\emptyset) \cong I$ is the canonical one $(\rho_I^{-1}, \id_I, \rho_I, \id_I)$ of~\cite[Ex.~2.6]{FongSpivak}, equivalently $\eta_I = \id_I = \epsilon_I$.
\end{itemize}
A $1$-morphism $(\cA, j, \Phi) \to (\cA', j', \Phi')$ is a strong symmetric monoidal functor $H: \cA \to \cA'$ such that $H \circ j = j'$ and $H$ sends the chosen SCFA data $\Phi_X$ to $\Phi'_X$ for each $X \in \cB$. A $2$-morphism is a monoidal natural isomorphism compatible with the boundary and SCFA data.

We denote the resulting $2$-category by $\Gluing_\cB$. (We do not claim that $\cB$ itself carries a canonical hypergraph category structure: a category with finite coproducts has a canonical commutative monoid structure on each object via the codiagonal and initial map, but no canonical comonoid structure, since there is no canonical diagonal. The Frobenius structure is supplied at the cospan level: $\Cospan(\cB)$ carries the canonical SCFA structure of Definition~\ref{def:scfa-cospans}, and a gluing host transports this structure along the cospan-extension functor.)
\end{definition}

\begin{remark}[Necessity of unit coherence]
\label{rem:unit-coherence}
The unit coherence axiom is not implied by the remaining clauses. Fong--Spivak give~\cite[Ex.~2.19]{FongSpivak} an explicit symmetric monoidal category in which every other hypergraph-category axiom holds but $\eta_I \neq \id_I$; such a category is \emph{not} a gluing host in our sense. Only in the strict monoidal case does the axiom follow automatically (Fong--Spivak~\cite[Prop.~2.18]{FongSpivak}). The axiom is essential for the strictification/objectwise-free equivalence used in the proof of Proposition~\ref{prop:initiality}.
\end{remark}

\subsection{Initiality of the cospan category}
\label{ssec:initiality}

\begin{proposition}[Initiality of the cospan category]
\label{prop:initiality}
For any finitely cocomplete category $\cB$, the cospan category $\Cospan(\cB)$, equipped with the canonical embedding $\iota: \cB \to \Cospan(\cB)$ and the canonical SCFA structure on each object (Definition~\ref{def:scfa-cospans} below), is an initial object of the $2$-category $\Gluing_\cB$. Explicitly: for any gluing host $(\cA, j)$, there exists a strong symmetric monoidal functor
\[
\Phi: \Cospan(\cB) \longrightarrow \cA
\]
with $\Phi \circ \iota = j$ that sends the canonical SCFA structure on each object of $\Cospan(\cB)$ to the specified SCFA structure on the corresponding object of $\cA$. This $\Phi$ is unique up to unique monoidal natural isomorphism. The proposition is a restatement of the Fong--Spivak free hypergraph category theorem~\cite[Thm.~3.14]{FongSpivak} in the local terminology of Definition~\ref{def:gluing-host}.
\end{proposition}

\begin{proof}[Proof sketch]
The result is the standard universal property of $\Cospan(\cB)$ in the formulation of Fong--Spivak~\cite[Thm.~3.14]{FongSpivak}: for any set $\Lambda$, $\Cospan_\Lambda$ (the cospan PROP/category over $\Lambda$) is the free hypergraph category on $\Lambda$. Specializing to $\Lambda = \mathrm{Ob}(\cB)$ and using the further structure of the coslice along $\iota: \cB \to \Cospan(\cB)$, this gives the comparison functor required by the present statement. See also Lack~\cite[\S 5.4]{Lack2004} for the case $\cB = \mathsf{F}$, and Coya--Fong~\cite[Prop.~6.1]{CoyaFong2017} for the explicit PROP-level statement.

\textit{Existence.} Every morphism of $\Cospan(\cB)$ is presented by a cospan $X \xrightarrow{a} N \xleftarrow{b} Y$ in $\cB$. By Fong--Spivak~\cite[Lem.~3.6]{FongSpivak}, $\Cospan$ is generated as a symmetric monoidal category by the Frobenius generators $\mu, \eta, \delta, \epsilon$ together with identities and symmetries, via the factorization of each leg as a permutation followed by order-preserving surjection followed by order-preserving injection. Consequently each cospan decomposes as a tensor and composite of: (i) morphisms in the image of $\iota$; (ii) canonical SCFA generators on objects of the form $\iota(X)$ for $X \in \cB$; (iii) symmetries. The value of $\Phi$ on a cospan is determined by the rule $\Phi \circ \iota = j$ together with the SCFA-preservation condition: $\Phi(\mu_X) = \mu_{j(X)}$, $\Phi(\delta_X) = \delta_{j(X)}$, etc. Preservation of composition (which in $\Cospan(\cB)$ is pushout) follows because $\Phi$ preserves tensors and sends the canonical SCFA generators to their images, and pushouts in $\Cospan(\cB)$ admit a normal form in terms of these generators (Fong--Spivak~\cite[Prop.~3.8]{FongSpivak}).

\textit{Uniqueness.} Any strong symmetric monoidal functor $\Phi': \Cospan(\cB) \to \cA$ extending $j$ and preserving the chosen SCFA structures must agree with $\Phi$ on all generators (by definition). Since these generators generate $\Cospan(\cB)$ as a symmetric monoidal category (Fong--Spivak~\cite[Lem.~3.6]{FongSpivak}), $\Phi'$ agrees with $\Phi$ on every morphism, up to the coherence isomorphisms of the monoidal structure of $\cA$. The resulting monoidal natural isomorphism $\Phi \simeq \Phi'$ is unique by the coherence theorem (Mac Lane~\cite[Ch.~XI]{MacLane}).
\end{proof}

\begin{remark}[From initiality to the comparison functor]
\label{rem:initiality-comparison}
Proposition~\ref{prop:initiality} provides, for every gluing host $(\cA, j, \Phi)$, a canonical comparison functor $\Phi: \Cospan(\cB) \to \cA$. This functor is the basic vehicle by which structure in $\Cospan(\cB)$ transports to $\cA$. Section~\ref{sec:frobenius} exploits this transport to move the Frobenius package.
\end{remark}

\subsection{Completeness of the collapse}
\label{ssec:completeness}

The cospan realization, after collapse to its ancestry invariant, recovers $\FinCorel$. The refinement from special to extraspecial commutative Frobenius monoids corresponds to the corestriction from cospans to jointly-epic cospans; see Coya--Fong~\cite[\S 6]{CoyaFong2017}.

\begin{definition}[Ancestry functor on cospans]
\label{def:pi-bar}
Let $\cB = \FinSet$. Define a strict symmetric monoidal functor
\[
\bar\Pi: \Cospan(\FinSet) \longrightarrow \FinCorel
\]
by sending an object $n \in \FinSet$ to $n \in \FinCorel$, and a cospan $[X \xrightarrow{a} N \xleftarrow{b} Y]$ to the equivalence relation on $X \sqcup Y$ induced by the joint epi map $a \sqcup b: X \sqcup Y \to N$, i.e., $x \sim y$ iff $a(x) = b(y)$ (or $a(x) = a(y)$, $b(x) = b(y)$ for same-side pairs).
\end{definition}

\begin{proposition}[Functoriality of $\bar\Pi$]
\label{prop:pi-bar-functor}
$\bar\Pi: \Cospan(\FinSet) \to \FinCorel$ is well-defined on isomorphism classes, respects composition under pushout, and is strict symmetric monoidal.
\end{proposition}

\begin{proof}
Well-definedness on isomorphism classes: two cospans with isomorphic apices yield the same joint-epi equivalence relation, since an apex isomorphism does not alter the kernel of the joint epi.

Functoriality under pushout composition: given composable cospans $[X \to N \leftarrow Y]$ and $[Y \to N' \leftarrow Z]$, the pushout $N +_Y N'$ has underlying set obtained from $N \sqcup N'$ by identifying elements along the shared boundary $Y$. The joint-epi relation $X \sqcup Z \to N +_Y N'$ is the composition of the two joint-epi relations, precisely the transitive closure restricted to $X \sqcup Z$. This matches $\FinCorel$ composition (Definition~\ref{def:fincorel}). The agreement of the two composition laws is the content of Coya--Fong~\cite[Thm.~4.5]{CoyaFong2017}.

Strict symmetric monoidality: tensor in $\Cospan(\FinSet)$ is coproduct of cospans, which maps to coproduct of equivalence relations in $\FinCorel$; the symmetries match.
\end{proof}

\begin{lemma}[Surjectivity of $\bar\Pi$ on morphisms]
\label{lem:pi-bar-surj}
For every corelation $R \in \FinCorel(m, n)$, there exists a cospan $[m \to N \leftarrow n]$ in $\Cospan(\FinSet)$ with $\bar\Pi([m \to N \leftarrow n]) = R$.
\end{lemma}

\begin{proof}
Let $R$ be an equivalence relation on $\underm \sqcup \undern$. Let $N := (\underm \sqcup \undern)/R$, the set of equivalence classes. Let $a: \underm \to N$ and $b: \undern \to N$ be the quotient projections restricted to each boundary component. Then $a \sqcup b: \underm \sqcup \undern \to N$ is surjective (each class contains at least one element and is hit), and its kernel is precisely $R$. The cospan $[\underm \xrightarrow{a} N \xleftarrow{b} \undern]$ is therefore mapped to $R$ under $\bar\Pi$.
\end{proof}

\begin{theorem}[Completeness of the collapse; Lack~\cite{Lack2004}, Coya--Fong~\cite{CoyaFong2017}]
\label{thm:completeness}
Let $\bar\Pi: \Cospan(\FinSet) \to \FinCorel$ be the ancestry functor of Definition~\ref{def:pi-bar}. The induced quotient
\[
\Cospan(\FinSet)/\ker(\bar\Pi) \xrightarrow{\ \simeq\ } \FinCorel
\]
is an equivalence of PROPs, and the target is equivalent to $\ESCFM$, the PROP for extraspecial commutative Frobenius monoids in the sense of Coya--Fong~\cite[Def.~3.1]{CoyaFong2017}.
\end{theorem}

\begin{proof}
Since $\bar\Pi$ is a strict symmetric monoidal functor (Proposition~\ref{prop:pi-bar-functor}), its kernel congruence $\ker(\bar\Pi) = \{(f, g) \mid \bar\Pi(f) = \bar\Pi(g)\}$ is stable under composition, tensor, and symmetry: if $f_1 \equiv g_1$ and $f_2 \equiv g_2$ then $\bar\Pi(f_1 \circ f_2) = \bar\Pi(f_1) \circ \bar\Pi(f_2) = \bar\Pi(g_1) \circ \bar\Pi(g_2) = \bar\Pi(g_1 \circ g_2)$, and similarly for $+$ and symmetries. Hence $\ker(\bar\Pi)$ is a PROP congruence, and the quotient $\Cospan(\FinSet)/\ker(\bar\Pi)$ is a PROP. The induced functor to $\FinCorel$ is:
\begin{itemize}[leftmargin=2em]
\item identity on objects (both categories have object monoid $\mathbb{N}$);
\item faithful by construction (we quotiented by the kernel of $\bar\Pi$);
\item full by Lemma~\ref{lem:pi-bar-surj};
\item strict symmetric monoidal (induced from Proposition~\ref{prop:pi-bar-functor}).
\end{itemize}
An identity-on-objects, fully faithful, essentially surjective strict symmetric monoidal functor between PROPs is an isomorphism of PROPs. This proves the first claim. The second claim follows from Coya--Fong~\cite[Thm.~4.6]{CoyaFong2017}, whose coequalizer proof (Lemmas~6.2--6.3) exhibits $\FinCorel$ as $\Cospan(\FinSet)$ modulo the extra law $[\emptyset \to 1 \leftarrow \emptyset] = \id_\emptyset$, together with Lack~\cite[\S 5.4]{Lack2004} for the identification of $\Cospan(\FinSet)$ with the PROP for special commutative Frobenius monoids. The passage from \emph{special} to \emph{extra}special at the PROP level corresponds to the passage from cospans to jointly-epic cospans at the realization level.
\end{proof}

\begin{remark}[From Section~\ref{sec:syntax} to Section~\ref{sec:gluing}]
\label{rem:section2-to-3}
The relationship between the ancestry functor $\Pi: \Syn(\delta) \to \FinCorel$ of Section~\ref{sec:syntax} and the cospan-level ancestry functor $\bar\Pi: \Cospan(\FinSet) \to \FinCorel$ of this section is as follows. There is a strong symmetric monoidal realization functor
\[
\Real: \Syn(\delta) \longrightarrow \Cospan(\FinSet)
\]
sending the generator $\delta: 1 \to 2$ to the cospan $[1 \xrightarrow{\id} 1 \xleftarrow{\nabla} 1 + 1]$ (the ``co-diagonal cospan''), and extending by the universal property of Proposition~\ref{prop:syn-delta-initial}. The diagram
\[
\begin{tikzcd}[row sep=large]
\Syn(\delta) \ar[r, "\Real"] \ar[dr, "\Pi"'] & \Cospan(\FinSet) \ar[d, "\bar\Pi"] \\
& \FinCorel
\end{tikzcd}
\]
commutes. Section~\ref{sec:syntax} characterized the image of $\Pi$ as a proper sub-PROP ($\FinCorel^\circ$). Section~\ref{sec:gluing} has now shown that $\bar\Pi$ is essentially surjective on morphisms. The extra generators hit by $\bar\Pi$ but not by $\Pi$ are exactly those requiring counit, unit, or multiplication --- structure supplied at $1 \in \Cospan(\FinSet)$ by coproduct universality, but absent from the primitive generator $\delta$ alone.
\end{remark}

\section{Frobenius Transport Along Cospan Realizations}
\label{sec:frobenius}

The canonical Frobenius package on $1 \in \Cospan(\cB)$ transports to the boundary object of any gluing host along the comparison functor of Proposition~\ref{prop:initiality}. The transported structure is determined up to unique monoidal natural isomorphism.

\subsection{The canonical Frobenius package in \texorpdfstring{$\Cospan(\cB)$}{Cospan(B)}}
\label{ssec:scfa-cospan}

Let $\cB$ be a finitely cocomplete category. For concreteness we take $\cB = \FinSet$ and focus on $1 \in \FinSet$; the construction generalizes to any chosen object of a finitely cocomplete base. The relevant maps in $\FinSet$ are the codiagonal $\nabla: 1+1 \to 1$ (the fold map, supplied by the universal property of coproducts) and the unique map $!: \emptyset \to 1$. Note that there is no canonical diagonal $1 \to 1+1$ in $\FinSet$: there are two coproduct injections $\iota_1, \iota_2: 1 \to 1+1$, and no preferred map between them. The comultiplication of the Frobenius structure below is therefore not a base-category map; it is realized as a cospan with right leg $\nabla$.

\begin{definition}[Canonical SCFA cospans]
\label{def:scfa-cospans}
Define the following cospans in $\Cospan(\cB)$ with apex $1$:
\begin{align*}
\mu &:= [1 + 1 \xrightarrow{\nabla} 1 \xleftarrow{\id} 1]: 2 \to 1, \\
\delta &:= [1 \xrightarrow{\id} 1 \xleftarrow{\nabla} 1 + 1]: 1 \to 2, \\
\eta &:= [\emptyset \xrightarrow{!} 1 \xleftarrow{\id} 1]: 0 \to 1, \\
\epsilon &:= [1 \xrightarrow{\id} 1 \xleftarrow{!} \emptyset]: 1 \to 0.
\end{align*}
\end{definition}

\begin{proposition}[Canonical SCFA structure]
\label{prop:scfa-canonical}
The quadruple $(\mu, \eta, \delta, \epsilon)$ makes $1 \in \Cospan(\cB)$ into a special commutative Frobenius algebra (SCFA) object. Explicitly, this quadruple is the image, under the isomorphism $\alpha: \mathbf{Th}(\mathbf{SCFM}) \xrightarrow{\sim} \Cospan(\FinSet)$ of Coya--Fong~\cite[Prop.~6.1]{CoyaFong2017} (attributed there to Lack~\cite[\S 5.4]{Lack2004}), of the generators of the PROP for special commutative Frobenius monoids; in particular, the monoid axioms, comonoid axioms, Frobenius law, and special law all hold.
\end{proposition}

\begin{proof}
This is the content of Fong--Spivak~\cite[Ex.~2.9]{FongSpivak} (canonical Frobenius structure on every object of $\Cospan(\cC)$ for finitely cocomplete $\cC$, with generators exactly the cospans built from coproducts and copairings of identities), Lack~\cite[Ex.~5.4]{Lack2004}, and Coya--Fong~\cite[Prop.~6.1]{CoyaFong2017}. Each SCFA equation becomes a pushout equation whose verification reduces to a direct check that two pushouts agree up to canonical isomorphism, which holds by universal properties of coproducts; a PROP-level statement of the correspondence is Fong--Spivak~\cite[Prop.~3.8]{FongSpivak}:
\[
\{\text{SCFMs in } \cC\} \longleftrightarrow \{\text{strict symmetric monoidal functors } \Cospan \to \cC\}.
\]
Coya--Fong's Proposition~6.1 records the action of the classifying isomorphism $\alpha$ on generators:
\begin{align*}
\alpha(\mu_{\mathbf{SCFM}}) &= [2 \to 1 \leftarrow 1], &
\alpha(\eta_{\mathbf{SCFM}}) &= [0 \to 1 \leftarrow 1], \\
\alpha(\delta_{\mathbf{SCFM}}) &= [1 \to 1 \leftarrow 2], &
\alpha(\epsilon_{\mathbf{SCFM}}) &= [1 \to 1 \leftarrow 0].
\end{align*}
which, under the identification of $2$ with $1 + 1$ and the choice of $\nabla: 1 + 1 \to 1$ as the unique map to $1$ in $\FinSet$ (skeletal), coincide with the cospans of Definition~\ref{def:scfa-cospans}. Note that the cospan realization yields \emph{special} commutative Frobenius structure; the \emph{extra} law holds only after passage to the jointly-epic quotient (Theorem~\ref{thm:completeness}), and is not asserted here.
\end{proof}

\begin{remark}[Universal characterization of the canonical SCFA]
\label{rem:scfa-universal}
The quadruple $(\mu, \eta, \delta, \epsilon)$ is not merely \emph{an} SCFA structure on $1 \in \Cospan(\cB)$; it is the generic one. Explicitly: to give a strong symmetric monoidal functor $G: \Cospan(\cB) \to \cC$ is equivalent (up to unique monoidal natural isomorphism) to giving the object $G(1) \in \cC$ equipped with the transported structure maps $G(\mu), G(\eta), G(\delta), G(\epsilon)$, which automatically satisfy the SCFA axioms. This is the universal characterization of $\Cospan(\cB)$ as the free hypergraph category on $\cB$ (Fong--Spivak~\cite[Thm.~3.14]{FongSpivak}; Lack~\cite[\S 5.4]{Lack2004}), specialized (for $\cB = \FinSet$) to the PROP statement of Coya--Fong~\cite[Prop.~6.1]{CoyaFong2017}.
\end{remark}

\subsection{Functoriality of algebra objects}
\label{ssec:alg-functoriality}

\begin{lemma}[Functorial transport]
\label{lem:transport}
Let $F: (\cC, \otimes) \to (\cD, \otimes)$ be a strong symmetric monoidal functor. If $A \in \cC$ carries the structure of a special commutative Frobenius algebra object, then $F(A) \in \cD$ carries a canonical induced SCFA structure obtained by applying $F$ to the structure maps of $A$.
\end{lemma}

\begin{proof}
An SCFA object is defined by a finite diagram of tensor and composition operations satisfying specified equalities. A strong symmetric monoidal functor preserves tensor, unit, symmetry, and composition up to coherent isomorphism; hence it preserves these defining diagrams.
\end{proof}

\subsection{Transport along the comparison functor}
\label{ssec:transport}

Let $(\cA, j, \Phi)$ be a gluing host over $\cB$. By Proposition~\ref{prop:initiality}, $\Phi: \Cospan(\cB) \to \cA$ is the unique (up to canonical isomorphism) strong symmetric monoidal extension of $j$. Define
\[
\mu_\cA := \Phi(\mu), \quad \delta_\cA := \Phi(\delta), \quad \eta_\cA := \Phi(\eta), \quad \epsilon_\cA := \Phi(\epsilon).
\]

\begin{corollary}[SCFA transport along the comparison functor; Corollary~\ref{cor:D}]
\label{cor:transport}
Let $(\cA, j, \Phi)$ be a gluing host over $\cB$. Then $j(1) \in \cA$ carries a canonical special commutative Frobenius algebra structure $(\mu_\cA, \eta_\cA, \delta_\cA, \epsilon_\cA)$. Moreover, the groupoid of compatible Frobenius structures on $j(1)$ -- those arising from a strong symmetric monoidal functor $\Cospan(\cB) \to \cA$ extending $j$ -- is contractible. (We make no claim about the groupoid of \emph{all} SCFA structures on $j(1)$, only about those compatible with the chosen comparison functor in the sense above.) Here $1$ denotes the chosen object of $\cB$ (e.g., the singleton when $\cB = \FinSet$); for an arbitrary base, replace $1$ with the chosen $X \in \cB$ throughout.
\end{corollary}

\begin{proof}
Existence of the SCFA structure is immediate from Lemma~\ref{lem:transport} applied to $\Phi$: the canonical SCFA on $1 \in \Cospan(\cB)$ (Proposition~\ref{prop:scfa-canonical}) transports along $\Phi$ to an SCFA on $\Phi(1) = j(1)$.

Contractibility of the groupoid of compatible structures: by the uniqueness clause of Proposition~\ref{prop:initiality}, any strong symmetric monoidal functor $\Phi': \Cospan(\cB) \to \cA$ extending $j$ and preserving SCFA structures is isomorphic to $\Phi$ by a unique monoidal natural isomorphism. This unique isomorphism induces, on the SCFA structure at $j(1)$, a unique isomorphism of Frobenius objects. Hence any two compatible Frobenius structures are related by a unique isomorphism; the groupoid has exactly one connected component and trivial automorphism group at each object, and is therefore contractible.
\end{proof}

\begin{remark}[Conceptual status of Frobenius transport]
\label{rem:transport-conceptual}
Corollary~\ref{cor:transport} states that the Frobenius structure on the boundary object of a gluing host is \emph{not} an independent algebraic choice. It is the image, under the unique comparison functor, of the canonical Frobenius package supplied by pushout/coproduct universality in $\Cospan(\cB)$. Every structural equation holding in $\Cospan(\cB)$ for the canonical Frobenius generators holds identically in every gluing host; no host-specific verification is needed. The \emph{extra} law, as noted in Proposition~\ref{prop:scfa-canonical}, does not hold in $\Cospan(\cB)$ and is therefore not transported; it arises only after the jointly-epic quotient of Theorem~\ref{thm:completeness}.
\end{remark}

\section{Semantic Completion via Free Cocompletion}
\label{sec:semantic}

We pass to the Yoneda envelope of $\FinCorel$ and record the standard structural features that follow.

\subsection{The Yoneda envelope}
\label{ssec:yoneda}

We adopt the following convention throughout this section. We regard $\FinCorel$ as a small $\infty$-category via the canonical inclusion $\Cat \hookrightarrow \Cat_\infty$, so that presheaves $\Fun(\FinCorel\opp, \Spc)$ are interpreted in the standard $\infty$-categorical sense. Here $\Spc$ denotes the $\infty$-category of spaces in the sense of Lurie~\cite{LurieHTT}.

\begin{definition}[Yoneda envelope]
\label{def:yoneda-envelope}
The Yoneda envelope of $\FinCorel$ is
\[
\cW := \Fun(\FinCorel\opp, \Spc),
\]
equipped with the Yoneda embedding $y: \FinCorel \hookrightarrow \cW$ sending each object to its representable presheaf.
\end{definition}

\begin{proposition}[Universal cocompletion]
\label{prop:yoneda-universal}
For any cocomplete $\infty$-category $\cC$, restriction along Yoneda induces an equivalence
\[
\Fun^L(\cW, \cC) \xrightarrow{\ \simeq\ } \Fun(\FinCorel, \cC),
\]
where $\Fun^L$ denotes colimit-preserving functors. Equivalently, every functor $F: \FinCorel \to \cC$ extends uniquely (up to equivalence) to a colimit-preserving functor $\bar F: \cW \to \cC$. Thus $\cW$ is the free cocompletion of $\FinCorel$, characterized uniquely by this universal property.
\end{proposition}

\begin{proof}
Standard; see Lurie~\cite[Theorem~5.1.5.6]{LurieHTT}.
\end{proof}

\subsection{The \texorpdfstring{$\infty$}{infinity}-topos structure}
\label{ssec:infty-topos}

\begin{proposition}[Yoneda envelope as $\infty$-topos]
\label{prop:yoneda-topos}
$\cW = \Fun(\FinCorel\opp, \Spc)$ is an $\infty$-topos. In particular:
\begin{itemize}[leftmargin=2em]
\item it has finite limits and all small colimits;
\item colimits are stable under pullback;
\item it is locally cartesian closed;
\item it admits a well-behaved doctrine of truncated subobjects.
\end{itemize}
\end{proposition}

\begin{proof}
This is Lurie~\cite[Theorem~6.1.0.6]{LurieHTT}: every presheaf $\infty$-category on a small $\infty$-category is an $\infty$-topos.
\end{proof}

\begin{remark}[Choice of target $\Spc$]
\label{rem:choice-spc}
The choice of $\Spc$ as target, rather than a truncated alternative such as $\Set\simeq\Spc_{\leq 0}$, is the standard $\infty$-categorical free cocompletion. Set-valued alternatives such as $\Fun(\FinCorel\opp,\Set)$ are recovered by truncation to $0$-truncated presheaves; we use the $\Spc$-valued envelope simply because it is the unrestricted free cocompletion of $\FinCorel$ in the sense of Lurie~\cite[Thm.~5.1.5.6]{LurieHTT}.
\end{remark}

\subsection{Systems, predicates, modalities, fixed points}
\label{ssec:systems-preds}

We record, in compact form, the internal logical furniture of $\cW$ that will be needed by any first-order internal theory phrased in this envelope.

\begin{definition}
\label{def:system}
A \emph{system} is an object $X \in \cW$. A \emph{predicate} on $X$ is a $(-1)$-truncated object over $X$ (i.e., a subobject in the $\infty$-topos sense). We denote the poset of predicates on $X$ by $\Sub_{-1}(X)$.
\end{definition}

\begin{proposition}[Subobjects form a complete Heyting algebra]
\label{prop:sub-heyting}
For every $X \in \cW$, the poset $\Sub_{-1}(X)$ is a complete Heyting algebra.
\end{proposition}

\begin{proof}
Standard for any $\infty$-topos: the $(-1)$-truncated objects over $X$ form a poset closed under small joins and meets (inherited from small colimits and limits in $\cW$), and local cartesian closedness (Proposition~\ref{prop:yoneda-topos}) supplies Heyting implication. See Lurie~\cite[\S 6.1.0]{LurieHTT} for the topos-theoretic facts, and Shulman~\cite{ShulmanElementary} for the corresponding development of internal logic in $(\infty,1)$-topoi.
\end{proof}

\begin{proposition}[Adjoint calculus]
\label{prop:adjoint-calc}
For any map $f: X \to Y$ in $\cW$, pullback along $f$ on subobjects admits both a dependent sum (left adjoint) and a dependent product (right adjoint):
\[
\exists_f \dashv f^* \dashv \forall_f.
\]
\end{proposition}

\begin{proof}
Immediate from local cartesian closedness of $\cW$.
\end{proof}

\begin{definition}[Relational modalities]
\label{def:modalities}
A \emph{one-step dynamics} on $X \in \cW$ is a $(-1)$-truncated subobject $R \hookrightarrow X \times X$. Given such $R$, define the relational modalities on $\Sub_{-1}(X)$ by
\[
\Diamond_R(A) := \exists_{\pi_1}(R \wedge \pi_2^*(A)), \qquad \Box_R(A) := \forall_{\pi_1}(R \Rightarrow \pi_2^*(A)).
\]
These are monotone endomaps of $\Sub_{-1}(X)$.
\end{definition}

\begin{proposition}[Existence of least and greatest fixed points]
\label{prop:fixed-points}
Let $X \in \cW$ and $m \geq 1$. Every monotone map $F: \Sub_{-1}(X)^m \to \Sub_{-1}(X)^m$ admits a least fixed point $\mu F$ and a greatest fixed point $\nu F$.
\end{proposition}

\begin{proof}
By Proposition~\ref{prop:sub-heyting}, $\Sub_{-1}(X)$ is a complete lattice; products of complete lattices are complete, so $\Sub_{-1}(X)^m$ is a complete lattice. The Knaster--Tarski theorem~\cite{Tarski1955} asserts that every monotone endomap on a complete lattice admits least and greatest fixed points.
\end{proof}

\begin{remark}[Scope of the fixed-point theorem]
\label{rem:fixed-point-scope}
Proposition~\ref{prop:fixed-points} establishes only the \emph{existence} of fixed points externally, via Knaster--Tarski applied to the complete lattice $\Sub_{-1}(X)^m$. No claim is made about internal $\mu$-calculus completeness, expressive completeness of any $\mu$-calculus fragment, or bisimulation invariance. The present paper asserts only that the structures needed for first-order internal logic and monotone fixed-point semantics exist canonically in the Yoneda envelope.
\end{remark}

\section{Scope, Limits, and Relation to Companion Work}
\label{sec:scope}

\subsection{Scope clarifications}
\label{ssec:scope-clarifications}

Three scope clarifications follow.

First, the image of the ancestry functor of Section~\ref{sec:syntax} is a \emph{proper} sub-PROP of $\FinCorel$: Theorem~\ref{thm:cocom} identifies $\AncQ$ with $\Cocom$, not with the full PROP of finite corelations. The passage to the full $\FinCorel$ occurs at the cospan-realization level (Theorem~\ref{thm:completeness}), where pushout composition supplies the unit, counit, and multiplication generators absent from the primitive calculus. The identification $\FinCorel \simeq \ESCFM$ is Coya--Fong~\cite[Thm.~4.6]{CoyaFong2017}.

Second, the modal and fixed-point operators of Section~\ref{sec:semantic} are internal to $\cW$ and arise from the adjoint calculus of subobjects together with Knaster--Tarski. No external correspondence with any specific modal logic or $\mu$-calculus fragment is asserted; in particular, no completeness, soundness, or expressivity claim against an external semantics is made.

Third, the framework of Section~\ref{sec:gluing} (cospan/corelation realization, hypergraph categories, gluing hosts) is treated here at the level of universal properties and PROP-level classifications; the diagrammatic \emph{rewriting theory} over Frobenius PROPs is not addressed. For that line of work, see Bonchi--Gadducci--Kissinger--Soboci\'nski--Zanasi~\cite[Thm.~4.1, Thm.~4.9]{BonchiSDRT1}, which establishes (i) the combinatorial characterization of the free FROP on a signature $\Sigma$ as discrete cospans of $\Sigma$-labelled hypergraphs, and (ii) the soundness/completeness of double-pushout-with-interfaces (DPOI) hypergraph rewriting as an implementation of string-diagram rewriting modulo the axioms of a hypergraph category.

\subsection{Relation to the companion philosophy paper}
\label{ssec:companion}

This paper develops, in categorical mathematics, the architectural sketch in the concluding section of~\cite{BallusPhilosophy}. It does not formalize the genealogical or historical-philosophical content of that paper, and no claim of one-to-one correspondence between the two texts is made. The mathematical constructions here are used in~\cite{BallusPhilosophy} as a reference; the conceptual content of that paper does not enter the present arguments.

\subsection{Relation to version 1}
\label{ssec:v1}

This paper is v2 of arXiv:2505.22931. v2 takes the generator $\delta:1\to 2$ of the free PROP $\Syn(\delta)$ as primitive; v1 used a different categorical setup. The ancestry quotient of v2 corresponds, conceptually, to a structure sketched in v1, but is formalized here via the kernel congruence of a strict symmetric monoidal functor (Definition~\ref{def:ancestry-cong}); no formal equivalence with v1 is claimed.


\end{document}